\documentclass[12pt,reqno]{amsart}
\usepackage[dvips]{graphics}
\usepackage{amssymb}
\numberwithin{equation}{section}

\newtheorem{thm}{Theorem}[section]
\newtheorem*{thm*}{Theorem}
\newtheorem*{thmmain*}{MAIN THEOREM}
\newtheorem{lem}[thm]{Lemma}
\newtheorem{cor}[thm]{Corollary}
\newtheorem{prop}[thm]{Proposition}
\newtheorem*{prop*}{Proposition}

\theoremstyle{definition}

\theoremstyle{remark}
\newtheorem{rem}[thm]{Remark}
\newtheorem{ex}[thm]{Example}
\newtheorem{quest}[thm]{Question}

\newcommand{\tref}[1]{Theorem~\ref{#1}}
\newcommand{\cref}[1]{Corollary~\ref{#1}}
\newcommand{\pref}[1]{Proposition~\ref{#1}}

\newcommand{\lref}[1]{Lemma~\ref{#1}}

\def\Hom{{\mathop{\text{Hom}}}}

\begin{document}
\title{Riemannian foliations on contractible manifolds}

\author{Luis Florit}
\thanks{The first and fourth-named authors were partially supported by
CNPq, Brazil.}
\address{Luis A.~Florit, IMPA; Est.~Dona Castorina, 110 22460-320, Rio de
Janeiro, RJ, Brazil}
\email {luis@impa.br}
\author{Oliver Goertsches}
\address{Oliver Goertsches, Fachbereich Mathematik, Universit\"at
Hamburg; Bundesstra\ss e 55, 20146 Hamburg, Germany}
\email{oliver.goertsches@math.uni-hamburg.de}
\author{Alexander Lytchak}
\thanks{The third-named author was partially supported by a Heisenberg grant of the
DFG and by the SFB~``Groups, Geometry and Actions''}
\address{Alexander Lytchak, Mathematisches Institut, Universit\"at
K\"oln; Weyertal 86-90, 50931 K\"oln, Germany}
\email{alytchak@math.uni-koeln.de}
\author{Dirk T\"oben}
\address{Dirk T\"oben, Departamento de Matem\'atica, Universidade Federal
S\~ao Carlos; Rod. Washington Luiz, 13565-905, S\~ao Carlos, SP, Brazil}
\email{dirktoben@dm.ufscar.br}



\begin{abstract}
We prove that Riemannian foliations on complete contractible manifolds have a closed leaf, and that all leaves are closed if one closed leaf has a  finitely generated fundamental group. Under additional topological or geometric
assumptions we prove that the foliation is also simple.

\end{abstract}


\maketitle
\renewcommand{\theequation}{\arabic{section}.\arabic{equation}}

\pagenumbering{arabic}

\section{Introduction}
In \cite{ghys} E. Ghys proved very powerful approximation theorems for
Riemannian foliations on compact simply connected manifolds. As an
important application he deduced that for any Riemannian foliation on a
rational homology sphere with leaf dimension larger than one,  all leaves
must be compact rational homology spheres. Unfortunately, his methods (and the powerful methods developed
later in \cite{salem}) do not apply to Riemannian foliations on complete
non-compact Riemannian manifolds, essentially because Tischler's Theorem
(\cite{Tischler}) applies to compact manifolds only. But complete
non-compact manifolds are very important in Riemannian geometry, most
notably the Euclidean space, whose Riemannian foliations have been
studied in several papers (cf. \cite{gromollEuc1}, \cite{gromollEuc2}).

In this note, we use more abstract topological methods to prove a
non-compact analogue of the Theorem of Ghys on foliations of rational
homology spheres mentioned above. In order to formulate the result, we
say that a topological space is \emph{rationally contractible} if all of
its homotopy groups are torsion groups.

\begin{thm} \label{mainthm}
Let $M$ be a simply connected, rationally contractible, complete
Riemannian manifold and let $\mathcal F$ be a Riemannian foliation on
$M$.  Then  there exists a closed leaf which is   rationally contractible.
Moreover, if there exists a closed leaf with finitely generated fundamental group, then all leaves are closed    and the quotient space $M/{\mathcal F}$ is also rationally contractible.
\end{thm}

The main examples of rationally contractible spaces, essentially the only
class which naturally appears in geometry, are contractible spaces. One
might expect that, unlike the case of spheres where Riemannian foliations 
with exceptional leaves do exist, Riemannian foliations on contractible
manifolds are \emph{simple foliations}, i.e., they are given by
Riemannian submersions. We   can verify this under  additional
finiteness assumptions   (cf. Appendix \ref{secprel} for  more about $H$-spaces):

\begin{cor} \label{secondprop}
Let $\mathcal F$ be a Riemannian foliation on a complete, contractible
Riemannian manifold $M$.   Then  there exists at least one closed and
rationally contractible leaf $L$.  Its universal cover $\tilde L$
is a rationally contractible  $H$-space. Moreover, for any field $K$, the homology of $\tilde L$
with coefficients in $K$ is finitely generated.

In addition, if the integral homology of $\tilde L$ is finitely generated, then the foliation
$\mathcal F$ is simple, with base $M/\mathcal F$ and all leaves being contractible.
\end{cor}

The homotopy type of the  leaf $L$ in \cref{secondprop} seems to
be extremely bizarre if $L$ is not contractible. If either the dimension
or the codimension of $\mathcal F$ are small then such objects cannot
occur, and the foliation turns out to be simple, even in the general
rationally contractible case. The same conclusion can be drawn if  $\mathcal F$ is \emph{cobounded}, i.e., if  the whole manifold $M$ is contained in a tube of finite radius around one leaf.
 More precisely, we deduce from Theorem
\ref{mainthm}:

\begin{cor} \label{smalldimcodim}
Let $\mathcal F$ be a Riemannian foliation on a complete rationally
contractible Riemannian manifold $M$.   If $\mathcal F$ is cobounded  then $\mathcal F$ has only one leaf. If the dimension of the leaves of
$\mathcal F$ is at most $3$, or their codimension is at most $2$, then
$\mathcal F$ is a simple foliation.  
\end{cor}

In fact, we expected that non-contractible manifolds with the topological
properties described for $L$ in \cref{secondprop} do not exist at all.
Unfortunately, this is not true, as the following example due to William
Dwyer shows.

\begin{ex}\label{ex:bizarre}
Consider the canonical map $S^3 \to S^3 _{\mathbb Q}$ from $S^3$ to its
localization at $0$ (cf.~\cite{HatcherSpSeq}), and consider the homotopy fiber of
this map, the \emph{torsionification} $T(S^3)$. Then, $T(S^3)$ is the
loop space of the homotopy fiber of the map of the classifying space
$BS^3 = \mathbb H P ^{\infty}$ to $\mathbb H P ^{\infty} _{\mathbb Q}$,
its localization at $0$. One verifies that $T(S^3)$ satisfies all
algebraic conclusions from \cref{secondprop} and is homotopy equivalent
to a finite-dimensional manifold. \end{ex}

We have tried to keep the presentation as simple as possible. Using more results  from
the theory of $H$-spaces and equivariant cohomology, one can slightly improve the above statements.  Namely, one can show that in \cref{secondprop}, the universal covering
$\tilde L$ is a loop space and not only an $H$-space.   In \tref{mainthm} one can conclude that all leaves are rationally contractible, also in the non-finitely generated case. However,
new insights are needed to answer the following questions in full
generality:

\begin{quest}
Under the assumptions of \tref{mainthm}, is it true that all leaves are closed, even if the fundamental groups of the leaves are  not finitely generated?
\end{quest}

\begin{quest}\label{quest:simple}
Is it true that every Riemannian foliation on a complete contractible
Riemannian manifold is simple?
\end{quest}

The finiteness assumption in \cref{secondprop} can be verified
if the curvature is non-negative. More precisely, in this case a variant of the soul construction can be used  to reduce the problem to the cobounded case,
so that \cref{smalldimcodim} applies:

\begin{thm} \label{firstthm}
Every Riemannian foliation on a complete contractible Riemannian manifold
of non-negative sectional curvature is simple.
\end{thm}

We would like to mention that the above results were not known even for
the Euclidean space with its flat Euclidean metric (cf.~\cite{GWbook},
p.~148). Only in this flat case, we have a short direct proof of
\tref{firstthm} using \cite{Bol} instead of \tref{mainthm}. We present
this proof in Subsection \ref{Euclideancase}, for the sake of completeness.  This proof would also provide an
independent  approach to \tref{firstthm} avoiding \tref{mainthm}, once the following
geometric question that generalizes \cite{Bol} could be settled (cf. \lref{newgeom}):

\begin{quest} \label{newgeomquest}
Let $M$ be a complete non-negatively curved manifold with soul $S$. Let $\mathcal F$ be a  singular Riemannian foliation on $M$ with closed leaves and compact quotient.
 Is it true that the projection $M\to M/\mathcal F$ factorizes through the Sharafutdinov retraction  $M\to S$?
\end{quest}

We have not  found a geometric proof of \tref{mainthm} under
the natural assumption that $M$ has non-positive (even constant
negative!) curvature, the main geometric source of contractible
manifolds. Under the additional assumption that $(M,\mathcal F)$ is
invariant under a large group of isometries, a short proof of the
simplicity of the foliation is given in \cite{conjugate}. We would like
to formulate:

\begin{quest}
Is \tref{firstthm} true in non-positive curvature?
\end{quest}

Finally, we would like to mention that \tref{firstthm} reduces the
classification of Riemannian foliations on Euclidean space to the case of
simple foliations, i.e., metric fibrations in terms of
\cite{gromollEuc1}, \cite{gromollEuc2}. In \cite{gromollEuc1} such
fibrations were shown to be homogeneous if the  leaves have dimension at
most 3. In \cite{gromollEuc2} the homogeneity in all dimensions has been
claimed, however, this paper contains a serious gap in Section 3 as
Stefan Weil has pointed out (cf.~\cite{Weil}). We could validate the proof
only under the additional assumption that the foliation is
\emph{substantial} in the terminology of \cite{gromollEuc2}. Thus the
following basic question about Riemannian foliations seems to be open in
its full generality:

\begin{quest}
Is any Riemannian foliation on the Euclidean space homogeneous?
\end{quest}

The paper is structured as follows. In Section \ref{secmolino} we recall
Molino's construction that describes leaf closures of Riemannian
foliations. We obtain a fibration relating the topology of leaf closures
to the equivariant cohomology of a natural action of the orthogonal group
(\pref{genstr}). Then we derive a simple criterion forcing one leaf   or, under a finiteness assumption, all leaves  to be closed (\pref{ghysidea}). 
In Section \ref{secmain}, we
present the proof of \tref{mainthm}, modulo topological statements about
$H$-spaces and equivariant cohomology that are proven at the end of the
paper. In Section \ref{secfinite}, we show how \cref{secondprop},
\cref{smalldimcodim} and \tref{firstthm} follow from \tref{mainthm} and provide
a short geometric proof of \tref{firstthm} for Euclidean spaces. In
Appendix \ref{secprel} we recall the Theorems of Hopf and Borel about
finite $H$-spaces and present a minor extension to a non-finite
situation. In Appendix \ref{sec:techprop} we recall some basic facts
about equivariant cohomology and prove a small generalization of a well-known
result relating ranks of isotropy groups to the Krull dimension of
equivariant cohomology rings.

\vskip .3cm

{\bf Acknowledgements.} We thank William Dwyer for
providing Example \ref{ex:bizarre}.

\vskip 1cm

\section{General structure of leaf closures} \label{secmolino}
\subsection{Riemannian foliations and leaf closures}
A foliation $\mathcal F$  on a Riemannian manifold $M$ is called
a Riemannian foliation if any geodesic normal to a leaf remains
normal to all leaves it intersects.  (Thus, in terms of \cite{Molino},
we always assume that the Riemannian metric is bundle-like with respect to the foliation).  If the Riemannian metric is complete,
the only case which appears in this text, this implies (and is essentially equivalent) to the fact that all leaves are equidistant, i.e., the
distance function to any leaf is constant on any other leaf.  We refer to \cite{Molino} and will assume some aquaintance with Riemannian foliations.

 Let $\mathcal F$ be an $(n-k)$-dimensional Riemannian foliation on a
complete $n$-dimensional Riemannian manifold $M$.    For any leaf $L$ of $\mathcal F$,
the closure $\bar L$ is a smooth submanifold and these leaf closures define a  decomposition $\bar {\mathcal F}$ of $M$, called a \emph{singular Riemannian foliation}.   The quotient $M/\bar {\mathcal F}$ is a metric space, which is  compact if and only if $\mathcal F$ is
cobounded.


The restriction of $\mathcal F$ to any leaf closure $\bar L$ is a Riemannian foliation with dense leaves.   Riemannian foliations with dense leaves have been described completely
in \cite{haefliger}.   As a special case of his description we will use:

\begin{lem} \label{haefligerlem}
Let $\bar L$ be the closure of the leaf $L$ of the Riemannian foliation $\mathcal F$.
If $\pi_1 (\bar L)$ is a torsion group then $\bar L=L$.
If $\pi_1 (\bar L)$ is abelian then  there is a homomorphism $h: \pi_1 (\bar L) \to
\mathbb R ^l$ with dense image, where $l$ is the codimension of $L$
in $\bar L$.  Moreover, the lift of $\mathcal F$ to the universal cover of $\bar L$ is a simple foliation with quotient space
$\mathbb R^l$.
 \end{lem}

We will use the following easy observation several times along the paper.
\begin{prop} \label{contrleaf}
Assume that the foliation $\mathcal F$ has a closed contractible leaf $L$.
Then $\mathcal F$ is simple.
\end{prop}

\begin{proof}
 For any leaf $L_1$ in a neighborhood of $L$ there is a canonical projection $L_1\to L$
which is a covering map. Since $L$ is simply connected this covering map must
be a diffeomorphism.  Hence $L$ is a principal leaf.  Since the principal leaf is closed all
other leaves are closed as well and any other leaf $L_2$ is finitely covered by $L$.
Since $L$ is contractible, $L_2$ is the classifying space of its fundamental group.
But the classifying space of any non-trivial finite group is infinite-dimensional.
Hence the fundamental group of $L_2$ is trivial. Thus all leaves are principal leaves,
which just means that $\mathcal F$ is a simple foliation.
\end{proof}

\subsection{Molino's construction}
In this subsection, we briefly recall the Molino bundle, but refer to
\cite{Molino}, \cite{ghys}, \cite{haefliger} and \cite{GT} for more
details.

Let $\mathcal F$ be an $(n-k)$-dimensional Riemannian foliation on a
complete $n$-dimensional Riemannian manifold $M$. Assume that
$\mathcal F$ is transversally oriented. We consider the principal $G
=SO(k)$-fiber bundle $\pi:\hat M \to M$ of oriented transverse
orthonormal frames, called the \emph{Molino bundle}. The foliation
$\mathcal F$ admits a canonical lift to a $G$-equivariant,
$(n-k)$-dimensional Riemannian foliation ${\hat {\mathcal F}}$ on
$\hat M$. Molino showed that the closures of the leaves of ${\hat
{\mathcal F}}$ constitute a $G$-equivariant simple Riemannian
foliation~$\bar{\hat {\mathcal F}}$, which is therefore given by a
Riemannian submersion
$\rho : \hat M \to W:= \hat M/ \bar {\hat {\mathcal F}}$. We denote the
typical fiber of this submersion, i.e., the homeomorphism type of the
closures of leaves of ${\hat {\mathcal F}}$, by $N$.

The action of $G$ on $\hat M$ descends to an action of $G$ on $W$ in
such a way that $\rho$ becomes a $G$-equivariant fiber bundle. The
quotient space $W/G$ is canonically identified with the space of leaf
closures of $\mathcal F$, i.e., $W/G = M/\bar {\mathcal F}$. Let $N_0$ be a leaf of $\bar {\hat {\mathcal F}}$, hence a point in $W$.
Denote by $G_{N_0}$ the isotropy group of the point $\{ N_0\} \in W$, which is just the set of all elements in $G$ sending the submanifold $N_0\subset \hat M$ to itself.
The action of $G_{N_0}$ on $N_0$  is free and the quotient space $N_0 /G_{N_0}$
is the closure  $\bar L_0$ of a leaf $L_0$ of $\mathcal F$.
The action of $G_{N_0}$ on the submanifold $N_0 \subset \hat M$ has orbits transversal
to $\hat {\mathcal F}$.    Therefore we conclude from the preceding considerations:

\begin{lem}  \label{genclosed}
With the notations above, if $l$ denotes the codimension of $\hat {\mathcal {F}}$ in $N_0$, then $\dim (G_{N_0}) \leq  l$, and equality holds
if and only if the projection $\bar L_0$ of $N_0$ is a closed leaf of the original foliation $\mathcal F$.
\end{lem}

\subsection{Relation to equivariant cohomology}\label{seceqcohom}
For the $G$-spaces $\hat M$ and $W$ as above, we denote their Borel
constructions by $\hat M_G:=\hat M\times_G EG$ and $W_G=W\times_G EG$, respectively (cf.~\cite{AlldayPuppe} and Appendix
\ref{sec:techprop} below). By functoriality, we obtain a natural
fibration $\rho_G:\hat M_G\to W_G$ with typical fiber $N$. Since $G$ acts
freely on $\hat M$, we obtain a fibration $EG\to \hat M_G\to \hat M/G=M$
with contractible fiber $EG$, which implies that $M$ and $\hat M_G$ are
homotopy equivalent. Using the above notation we then have:

\begin{prop} \label{genstr}
The Borel construction $\hat M_G$ of the Molino bundle $\hat M$ is
homotopy equivalent to $M$, and the fibration $\hat M_G \to W_G$ has
fiber $N$.
\end{prop}

The foliation $\mathcal F$ is closed if and only if $\hat {\mathcal F}$
is closed. In this case the quotient space $B= M/\mathcal F$ is a
Riemannian orbifold and $M\to B$ a generalized Seifert fibration. The
leaves of $\hat {\mathcal F}$ are diffeomorphic via the projection
$\pi :\hat M \to M$ to the regular leaves $L$ of~$\mathcal F$. Moreover,
the Borel construction $W_G$ in \pref{genstr} coincides by definition
with Haefliger's \emph{classifying space} $\hat B$ of $B$
(\cite{classify}, Section 4). In particular, the projection
$W_G=\hat B \to B=W/G$ induces an isomorphism on rational (co)homologies.

\subsection{The simply connected case}

We now assume that the complete Riemannian manifold $M$ is simply
connected. Then the Riemannian foliation $\mathcal F$ is transversally
oriented and we have the whole structure described in the previous two
subsections.  Moreover, the leaf closures $N$ of $\hat {\mathcal F}$ have
abelian fundamental group.    In addition, the stabilizer $G_{N_0}$ of
any point $N_0 \in W$ has  trivial adjoint representation.
In particular, its identity component is a torus \cite[Lemma
4.6]{GT}.  We claim:

\begin{lem} \label{injectlemma}
With the notations above, let $T$ be the  connected component of
$G_{N_0}$. Then the map $\pi_1 (T) \to \pi_1 (N_0)$,  induced by
sending $T$ to an orbit of the action of $T$ on $N_0$  is an injection.
\end{lem}

\begin{proof}
Consider the lift of the action of $T$ on $N_0$ to the action
of the universal cover $\tilde T$ on $\tilde N_0$. By Lemma \ref{haefligerlem}, the foliation
$\hat {\mathcal F}$ lifts to a simple foliation on $\tilde N_0$ with
quotient $\mathbb R ^l$. Since the action of $T$ preserves the foliation and is transversal to it, we get an almost free isometric action of $\tilde T$ on $\mathbb R^l$.  Thus this action of $\tilde T$ on $\mathbb R^l$ must be free.

  Hence the action of $\tilde T$ on $\tilde N_0$ is free as well.
This directly implies that the kernel of the map
$\pi _1 (T) \to \pi_1 (N_0)$ is trivial.
\end{proof}

 Let~$l$ denote the codimension of $\hat {\mathcal F}$ in its
closure $\bar {\hat {\mathcal F}}$. Because $N$ is the fiber of the fibration $\hat M_G\to W_G$ with simply-connected total space, $\pi_1(N)$ is abelian.  Due to \lref{haefligerlem}, we have a homomorphism $h: \pi_1 (N) \to \mathbb R ^l$ with a dense image.
Hence, $\pi_1 (N) \otimes \mathbb Q$ is a $\mathbb Q$-vector space of dimension at least $l$.
If $\pi_1 (N)$ is finitely generated,  and the dimension of $\pi_1 (N) \otimes \mathbb Q$
is $l$ then   the image of $h$ would be a lattice in $\mathbb R ^l$, hence it could not
have a dense image, unless $l=0$.
We have shown:

\begin{lem} \label{injective}
We have the inequality $H^1 (N, \mathbb Q)  \geq l$.  In case of  equality, either
$l=0$ and the foliation $\mathcal F$ is closed, or the group $\pi_1 (N)$ is not finitely
generated and satisfies $\pi_1 (N) \otimes \mathbb Q = \mathbb Q^l$.
\end{lem}

Combining this lemma with \lref{genclosed} we obtain:

\begin{prop} \label{ghysidea}
For any   $N_0 \in W$ we have the inequality $\dim (G_{N_0} ) \leq \dim (H^1 (N_0 ,\mathbb Q))$.   If equality holds then  $N_0$  projects to a closed leaf of
$\mathcal F$.  If, in addition, $\pi_1 (N_0)$ is finitely generated then the foliation
$\mathcal F$ is a closed foliation.
\end{prop}

\vskip .5cm

\section{The main argument} \label{secmain}
We now give a proof of \tref{mainthm}, modulo some auxiliary topological
results which will be proven in the appendices. We continue to use the
notation introduced in the previous section, and assume that $M$ is
simply-connected and rationally contractible.

Consider the fibration
$\hat M_G \to W_G$ with fiber $N$. The space $\hat M _G$ is simply
connected, hence so is $W_G$ and the space $N$ is \emph{abelian}, i.e.,
its fundamental group is abelian and acts trivially on higher homotopy
groups (\cite{Hatcher}, p.~419, Exercise 10). The homotopy fiber of the
embedding map $N \to \hat M_G$ is the loop space $\Omega W_G$
(\cite{Hatcher}, p.~409). Since all homotopy groups of $\hat M_G$ are
torsion groups, the long exact homotopy sequence of this fibration
reveals that the map $p:\Omega W_G \to N$ induces isomorphisms on all
rational homotopy groups
$p_{\ast}: \pi_{\ast} ( \Omega W_G)
\otimes \mathbb Q \to \pi_{\ast} (N ) \otimes \mathbb Q$.

Since $N$ and $\Omega W_G$ are abelian topological spaces, the map $p$
induces an isomorphism of rational cohomology rings
$p^{\ast} :H^{\ast} (N, \mathbb Q) \to H^{\ast} (\Omega W_G ,\mathbb Q)$.
Thus the rational cohomology of the $H$-space $\Omega W_G$ (see Apendix
\ref{secprel} for more information on $H$-spaces) vanishes in degrees
larger than the dimension of $N$. We now use the following result, a
slight extension of a classical theorem of Hopf (\cite{Hatcher}, p.~285),
whose proof is postponed to Appendix \ref{secprel}, and deduce that the
rational cohomology ring of $N$ is the cohomology ring of a finite
product of spheres.

\begin{prop} \label{hopfspace}
Let $X$ be a connected $H$-space. Let $K$ be a field and assume that
$H^{\ast} (X,K)$ vanishes in all degrees larger than $m$. Then $H^{\ast}
(X,K)$ has dimension at most~$2^m$ over $K$.

Moreover, if $K$ has characteristic $0$, then $H^{\ast} (X,K)$ is
isomorphic to the cohomology ring of a product of spheres. Thus $H^{\ast}
(X,K) $ is the antisymmetric algebra
$\Lambda (y_1,\ldots,y_q)$ with $\deg y_i=:e_i$ odd.
\end{prop}

Given that $H^{\ast} (N,\mathbb Q)= \Lambda(y_1,\ldots,y_q)$ with
$\deg y_i=:e_i$ odd, and the assumption that
$H^{\ast} (\hat M_G, \mathbb Q)$ vanishes in positive degrees, the
transgression theorem of Borel (\cite{Borel}, Th\'eor\`eme 13.1, see also
\cite{MimuraToda}, Theorem VII.2.9) gives us that the cohomology ring
$H^{\ast} (W_G ,\mathbb Q)$ is a polynomial ring
$\mathbb Q[x_1,...,x_q]$, generated by elements $x_i$ of degree $e_i +1$.
In particular, $\dim H^2 (W_G, \mathbb Q)= \dim H^1 (N, \mathbb Q)$.
This dimension  is exactly the number of elements $y_i$   of degree $1$.

Now, we can apply the following result (essentially Theorem 7.7 from
\cite{Quillen}), whose proof will be explained in Appendix
\ref{sec:techprop}, in order to find a lower bound on the ranks and thus
also bounds on the dimension of stabilizers of the action of $G$ on $W$.

\begin{prop} \label{prop:eqcohinjection}
Let $G$ be a compact Lie group that acts smoothly on a manifold~$X$.
Assume that the $G$-equivariant cohomology ring $H^{\ast} (X_G)$ is a
polynomial ring in $q$ variables. Then there is a point $x\in X$ such
that the rank of the stabilizer $G_x$ is at least~$q$.
\end{prop}

Applying this proposition to our action of $G =SO(k)$ on $W$, we find
some point $N_0\in W$ such that its stabilizer $G_{N_0}$ has rank and
thus dimension at least $q$. On the other hand, we know that
$q \geq \dim H^2 (W_G ,\mathbb Q) =\dim H^1 (N_0, \mathbb Q)$.
We apply \pref{ghysidea}  and deduce that $q=\dim H^1 (N_0 ,\mathbb Q)$
and  that $N_0$ projects to a closed leaf $L_0$ of $\mathcal F$ in $M$.

Now we are going
to analyze the topology of this closed leaf $L_0$.  Since $q=\dim H^1 (N_0 ,\mathbb Q)$, all generators $y_i$ of the cohomology of $N_0$ have degree $1$
and $N_0$ has the rational cohomology ring of a $q$-dimensional torus.

We claim that all higher homotopy groups $\pi_j (N_0), j\geq 2$ of $N_0$ are torsion groups.     Indeed, let $E=\mathbb S^1_\mathbb Q$ be the Eilenberg MacLane space
$K(\mathbb Q ,1)$, which is just the rationalization of the circle $\mathbb S^1$.
We have $H^{\ast} (E, \mathbb Q)= H^{\ast} (S^1, \mathbb Q)$.
This space $E$ is the classifying space for $H^1(\cdot, \mathbb Q)$, hence we find
maps $f_1,....., f_q :N_0 \to   E$,  such that the pull-backs of the
canonical generator $e$ of $H^1 (E,\mathbb Q)$ are exactly   $y_i=f_i ^{\ast}  (e)$.
The knowledge of the rational cohomology ring of $N_0$ now tells us
that the map $F=(f_1,...,f_q) :N_0 \to E^q$ induces an isomorphism on all
rational cohomology groups.   Since $E$ and $N_0$ are abelian spaces,
the map $F$ induces isomorphisms on all rational homotopy groups.
But all higher homotopy groups of $E$ vanish, which  proves the claim.

Now we consider the fiber bundle $N_0 \to L_0$ with fiber $G_{N_0}$.
Due to \lref{injectlemma}, the orbit map $G_{N_0} \to N_0$
induces an injection on $\pi_1$.  Since $\pi_1 (G_{N_0})$ and $\pi_1 (N_0)$ tensored with
$\mathbb Q$ are both isomorphic to $\mathbb Q ^l$,   the cokernel of the injection must be a torsion group.  Since $\pi_0 (G_{N_0})$ is finite, we deduce that    $\pi_1 (L_0)$ is
a torsion group.
Since all higher   homotopy groups of $N_0$ and $G_{N_0}$ are torsion groups, we deduce from the long exact sequence in homotopy that
all higher homotopy groups of $L_0$ are torsion as well.
Thus $L_0$ is a rationally contractible space.

  Note further that $\pi_1 (L_0)$ contains the abelian image of $\pi_1 (N_0)$
as a subgroup of finite index. Considering the same fiber bundle as above, $\pi_1(L_0)$ is finitely generated if and only if $\pi_1(N_0)$ is finitely generated, and this is in turn equivalent to any other leaf closure in $M$ having finitely generated fundamental group. Assuming this in addition, $\pi_1 (L_0)$ must be a finite group.     The principal leaf $L$ of $\mathcal F$ near $L_0$  admits a canonical covering map  to $L_0$. Since this covering must be finite, due to the finiteness of $\pi_1 (L_0)$, we see that $L$ is closed as well. Then all leaves are closed.
Since $L$ finitely covers $L_0$ and any other leaf is finitely covered by $L$
all leaves are rationally contractible.

Once we know that $\mathcal F$ is closed, the quotient $B= M/\mathcal F$
is a Riemannian orbifold. Moreover, the Borel construction $W_G$ is
Haefliger's classifying space of $B$, in particular, its rational
cohomology groups are non-zero only in finitely many degrees, since they
coincide with the rational cohomology groups of $B$. As we already know
that $H^\ast(W_G,\mathbb Q)$ is a polynomial ring, this tells us that the
rational cohomology groups of $W_G$ vanish in positive degrees. Since
$W_G$ is simply connected, $W_G$ is a rationally contractible space. From
the long exact homotopy sequence of the fibration $\hat M _G \to W_G$, we
conclude that~$N$ is rationally contractible as well.

Finally, $B$ is simply connected and its rational cohomology groups
vanish in positive degrees, since they coincide with the rational
cohomology groups of $W_G$. Therefore $B$ is rationally contractible.

\vskip .5cm

\section{Finiteness conditions} \label{secfinite}
We continue to use notations introduced in Section \ref{secmolino}.
\subsection{Finite homology:} \label{finitehom}
{\it Proof of \cref{secondprop}. }
Using \tref{mainthm} we see that there is a closed rationally contractible leaf $L$. Hence its universal covering is rationally
contractible as well.

The manifold $N_0$ appearing in the proof of \tref{mainthm}
is the fiber of the fibration $\hat M_G \to W_G$. Since $\hat M_G$ is contractible, $N$ is homotopy
equivalent to the loop space $\Omega W_G$. In particular, $N_0$ is an $H$-space, hence so is its universal covering $\tilde N_0$.

 We have the fibration $N_0\to L$ with fiber $G_{N_0}$.  In the course of the proof of \tref{mainthm}, we have seen
 that the induced map $\pi_1 (G_{N_0}) \to \pi_1 (N_0)$ is injective.
 Therefore, the  fiber of the fibration  $\tilde N_0 \to \tilde L$
 between the universal coverings is the universal covering of the connected component of $G_{N_0}$.  Since the connected component of
 $G_{N_0}$ is a torus, we deduce that $\tilde L$ is homotopy equivalent to $\tilde N_0$.  Thus $\tilde L$ is an $H$-space.

 We conclude from \pref{hopfspace} that the
homology groups of $\tilde L$ with coefficients in any field $K$ are finite
dimensional $K$-vector spaces.

If the homology of $\tilde L$ with integral coefficients is finitely generated, then a
theorem of Browder (\cite{Browder0}, see \cref{browdercon} below) tells
us that the rationally contractible $H$-space~$ \tilde L$ must be contractible.   Then $L$ is the classifying space of its fundamental group, which is torsion.  Thus the fundamental group of $L$ is trivial and $L$ is contractible.
From \pref{contrleaf}  we get that  $\mathcal F$ is a simple foliation given by a Riemannian
submersion $M\to B$.

Since the fiber $L$ and the total space $M$ are contractible, the base
$B$ is contractible as well, due to the exact sequence in homotopy.
\qed

\subsection{Small dimension:} {\it Proof of \cref{smalldimcodim}. }
Assume first that the quotient $B= M/\bar {\mathcal F}$ is compact.
Since $B= W/G$, the space $W$ must be compact as well.
Hence $W$ has finitely generated integral homology.
Applying \cite{HatcherSpSeq}, Lemma 1.9, we see that the space $W_G$
has finitely generated integral homology in each dimension. Therefore, $W_G$
has finitely generated homotopy groups in each dimension.  From the fibration
$ M_G\to W_G$ we deduce that $\pi_1 (N)$ is finitely generated. Then the fundamental group of a closed leaf $L_0$ is finitely generated as well. Using \tref{mainthm},
we obtain that  $\mathcal F$ is closed. But then $B$ is a simply connected compact orbifold.
Hence its rational homology does not vanish in the maximal dimension.
Since $B$ is rationally contractible, this implies that $\dim (B)=0$. Hence $\mathcal F$
has only one leaf.

Assume now that the dimension of the leaves of $\mathcal F$ is at most
$3$, and let $L$ be a closed rationally contractible leaf.
 We claim that $L$ is contractible. If $\dim (L) \leq 2$,
this is a direct consequence of the classification of one and
two-dimensional manifolds. Assume $\dim (L)= 3$. Then the universal
covering $\tilde L$ is still rationally contractible. In particular,
$\tilde L$ is non-closed, hence $H_3 (\tilde L)=0$. By assumption,
$\pi_1 (\tilde L) =0$. From Poincar\'e duality we know that
$H_2 (\tilde L)$ is torsion-free. Since
$\pi_2 (\tilde L)=H_2 (\tilde L) $, and since $\tilde L$ is rationally
contractible, we obtain $H_2 (\tilde L)=0$. Therefore, $\tilde L$ is
contractible. Thus $L$ is the classifying space of its fundamental group
$\pi _1 (L)$, which is a torsion group. We conclude that $L=\tilde L$.
Now the simplicity of the foliation follows  from \pref{contrleaf}.

Let us assume now that the codimension of the leaves is at most $2$.
Assume first that the foliation $\mathcal F$ is closed. The quotient
$B=M/\mathcal F$ is an orbifold with trivial orbifold fundamental group
(which is just the fundamental group of $W_G$). The classification of one and
two-dimensional orbifolds (cf.~\cite{kleinerlott}, Section~2.3) shows
that either $B$ has no singularities and is diffeomorphic to the
Euclidean space, or its underlying topological space is the two sphere.
But the second case cannot occur, since we already know that $B$ cannot be compact.
 Thus $B$ has no singularities. Hence
$\mathcal F$ has no exceptional leaves and $\mathcal F$ is simple.

Assume now that $\mathcal F$ is non-closed.  Then $M/\bar {\mathcal F}$ has dimension
less than two.  Due to \lref{haefligerlem},  $\mathcal F$ cannot have dense leaves
in $M$, unless $\mathcal F$ has only one leaf, hence we may assume that
$\mathcal F$ has codimension $2$ and $\bar {\mathcal F}$ has codimension $1$.
The closed leaves of $\mathcal F$ are exactly the singular leaves of $\bar {\mathcal F}$.
Since at least one such leaf exists, the quotient $M/\bar {\mathcal F}$ is either an interval
or a ray. If the quotient is a ray, then $M$ retracts to the closed leaf of $\mathcal F$,
which has an infinitely generated fundamental group by \tref{mainthm}. This is impossible, since $M$ is simply connected. Hence the quotient $B$ is a compact interval.
But $B$ cannot be compact, unless it is a point.

\qed

\subsection{Non-negative curvature:} {\it Proof of \tref{firstthm}.}
Due to \pref{contrleaf},  it suffices to find a closed contractible leaf.  Following  the ideas
of \cite{gromollEuc1} and  \cite{Bol}, Section 2.1, we are going to find a \emph{totally convex} leaf.
Recall that a subset $X'$ of $M$  is called totally convex, if it contains any  (not necessarily minimal) geodesic connecting any pair of its points.   A closed totally convex subset
is a topological manifold with boundary, whose set of inner points is a totally geodesic submanifold.     If a closed totally convex subset $X$ is a manifold (without boundary)
then $X$ is homotopy equivalent to $M$.

Assume now that we can find a closed totally convex $\mathcal F$-saturated submanifold $M'$ of  $M$, such that the restriction of $\mathcal F$ to $M'$ is cobounded.  Then $M'$ is contractible and by \cref{smalldimcodim}, $M'$ must be a leaf of $\mathcal F$. Hence,  \tref{firstthm} follows directly from  the following general geometric observation, whose proof, essentially contained in \cite{Bol}, we shortly recall below:

\begin{lem} \label{newgeom}
Let $M$ be a (not necessarily contractible) complete Riemannian manifold of non-negative curvature with a Riemannian foliation $\mathcal F$.  Then there exists a closed, totally convex $\mathcal F$-saturated submanifold $M'$ of $M$ such that the restriction  of $\mathcal F$ to  $M'$ is cobounded.
\end{lem}

\begin{proof}
  Let $\mathcal C$ be the set of all closed, totally  convex, $\mathcal F$-saturated subsets of $M$.  The set $\mathcal C$ is non-empty, since it contains $M$,
and it  is closed under  intersections.   Let us fix some set $X \in \mathcal C$ of smallest dimension in
$\mathcal C$.  Fix some $x\in X$ and let $L$ be the leaf through $x$. We may assume that $X$ is the smallest element of $\mathcal C$ which contains $L$.

  We claim that $X$ is a  submanifold and that $\mathcal F$
is  cobounded  on $X$. First, assume  that $\mathcal F$ is   not cobounded.  Hence we find   leaves
$L_i \subset X$ running away from $x$, i.e., such that $l_i=d(x,L_i) $ converges to infinity.  Consider the normalized distance  functions $f_i:M\to \mathbb R$ to the subset $L_i$  given by $f_i (y)=d(y,L_i)-l_i$.   The functions  $f_i$ vanish at $x$, are $1$-Lipschitz and $\mathcal F$-basic, i.e., constant on leaves of $\mathcal F$.  Consider a pointwise limit $f$ of a subsequence of $f_i$.
  Since the sets $L_i$ run to infinity, the function $f$ is concave by Toponogov's theorem
(see the usual  proof of the soul theorem).
Moreover, $f$ is non-constant on $X$ since it has velocity $1$ on a ray starting at $x$ and being a limit of shortest geodesics from $x$ to $\bar L_i$.  Hence the    superlevel set
$f^{-1} ([0, \infty ))\cap X$   is a closed totally convex subset of $M$, which is $\mathcal F$-saturated, contains $L$  and is properly contained in $X$. This contradicts the minimality of $X$ and thus we may assume that  $\mathcal F$ is cobounded on $X$. 

Assume now  that $X$ is a manifold with non-empty boundary $\partial X$.
The boundary $\partial X$ must be $\mathcal F$-saturated as well, as we  immediately see in the local picture (see \cite{Bol}, for a slightly more difficult argument in the case of singular foliations).    The  distance function $d_{\partial X}$ to the boundary
is a concave function on $X$  (see the usual soul construction), hence its superlevel sets are closed totally convex   subsets of $X$.   Now, by the coboundedness, the distance function $d_{\partial X}$ assumes a maximum on $X$. And the set $M'$ of points where this maximum is attained is a
convex set of dimension less than the dimension of $X$, contradicting to the choice of $X$.
\end{proof}

\begin{rem}
In connection with Question \ref{newgeomquest}, we mention that as in \cite{Bol}, the above proof applies without changes to singular Riemannian foliations (and more generally transnormal decompositions)  on non-negatively curved manifolds.
\end{rem}

\subsection{The Euclidean case} \label{Euclideancase} Here we present a
short geometric proof of \tref{firstthm} for the  Euclidean space.

Namely, consider the closed singular Riemannian foliation
$\bar {\mathcal F}$ on $M = \mathbb R^n$. Due to \cite{Bol},
Theorem 2.1
(cf.~\cite{conjugate}, Corollary 1.3 for a simplification of the
argument, starting from \lref{newgeom} above), there is at least one leaf $\bar L$ of $\bar {\mathcal F}$
which is an affine subspace of $\mathbb R^n$.
Since an affine subspace is
contractible,  this leaf closure $\bar L$ must be in fact a leaf $L=\bar L$
of~$\mathcal F$ due to  \lref{haefligerlem}.  The result now follows from \pref{contrleaf}.

\vskip .5cm
\vskip .5cm

\appendix

\section{Preliminaries on $H$-spaces} \label{secprel}

We refer to \cite{Hatcher}, pp.~281-292, for this section, in which we
assume that all spaces have the homotopy type of a CW complex. Recall
that an \emph{$H$-space} is a topological space together with a
"multiplication" $\mu:X\times X\to X$ that has a "canonical unit
element". Any $H$-space is an abelian topological space (\cite{Hatcher},
Example 4A.3) and its universal covering is again an $H$-space.
In our applications, the space $X$ will be the loop space~$\Omega Z$ of
some other space $Z$, equipped with the multiplication given by
concatenation of loops.

For a ring $R$, a \emph{Hopf algebra} over $R$ is a graded $R$-algebra
$A$ (always associative and graded-commutative) together with a graded
homomorphism $\mu:A\to A\otimes _R A$, called the comultiplication, that
satisfies an algebraic equivalent of the unity axiom for $H$-spaces; see
\cite{Hatcher}, p.~283. In the sequel, all Hopf algebras will be
\emph{connected}, i.e., satisfy $A_0 =R$.

The structure of finite-dimensional Hopf algebras over fields is
essentially known due to the following theorem proved by Hopf in
characteristic $0$ and by Borel in positive characteristics
(\cite{Hatcher}, p.~285 or \cite{Milnor}, Theorem 7.11):

\begin{thm} \label{hopfborel}
Let $A$ be a finite-dimensional Hopf algebra over a perfect field $K$.
Then~$A$ is a tensor product of Hopf algebras $A_i$ over $K$, each
generated by one element. Moreover, if $K$ has characteristic $0$, each
$A_i$ has dimension $2$, and $A$ is isomorphic to the cohomology ring
with coefficients in $K$ of a product of odd-dimensional spheres.
\end{thm}

As a consequence we have:
\begin{cor} \label{dimensionbound}
Let $A$ be a finite-dimensional Hopf algebra over a perfect field $K$.
Then the dimension of $A$ is bounded from above by $2^d$, where $d$ is
the maximal degree in which the graded algebra $A$ is not $0$.
\end{cor}

\begin{proof}
The result clearly holds for algebras generated by one element. Since the
bound is preserved under tensor products, the statement follows from the
above theorem of Borel and Hopf.
\end{proof}

If $X$ is an $H$-space with multiplication $\mu$ and if the cohomology
ring $A= H^{\ast} (X, K)$ with coefficients in a field $K$ is
finite-dimensional, then $\mu$ induces a comultiplication on~$A$ that
makes it into a Hopf algebra. Using this, the above theorem of Borel and
Hopf and the Bockstein sequence, it is not difficult to get the following
result of Browder (Corollary 7.2 in \cite{Browder0}):

\begin{cor} \label{browdercon}
Let $X$ be a connected $H$-space, whose integer homology $H_{\ast} (X,
\mathbb Z)$ is finitely generated. Then $H_d (X, \mathbb Z) =\mathbb Z$,
where $d$ is the maximal degree, in which the homology does not vanish.
In particular, if $X$ is rationally contractible, it is in fact
contractible.
\end{cor}

The $H$-spaces $X$ arising in the proof of \tref{mainthm}, a priori, do
not satisfy the assumption of the last corollary. Instead, their
cohomology is the cohomology of non-compact finite dimensional manifolds,
in particular, their cohomology groups vanish in all but finitely many
degrees.

We are going to prove \pref{hopfspace} now, showing that in this more
general situation the cohomologies with coefficients in a field do not
show anomalies and behave like in the finite case.

\begin{proof} [Proof of \pref{hopfspace}]
By the universal coefficient theorem, it suffices to prove the result for
$K=\mathbb Q$ and $K=\mathbb F_p$, for all prime numbers~$p$. Thus let
$K$ be one of these fields and let us consider homologies and
cohomologies with coefficients in $K$ until the end of the proof. Hence
the field $K$ is a perfect field and we may use \tref{hopfborel} and
\cref{dimensionbound}.

We denote by $A$ the cohomology ring $H^{\ast} (X, K)$. Since any
finite-dimensional Hopf algebra contained in $A$ has dimension at most
$2^m$ by \cref{dimensionbound}, and since any Hopf algebra can be
exhausted by finite-dimensional subalgebras (cf.~\cite{Milnor}), we
conclude that any Hopf subalgebra of $A$ has finite dimension, bounded
above by $2^m$.

It remains to prove that $A$ is a Hopf algebra. Unfortunately, the
homology is not finitely generated, and the assumption that $X$ is an
$H$-space does not provide $A$ automatically with the structure of a Hopf
algebra. Instead we argue as follows: The map $\mu$ defines a
homomorphism $\mu ^{\ast} : H^{\ast} (X) \to H^{\ast} (X\times X)$. We
consider the K\"unneth formula with coefficients in a field (see
\cite{Hatcher}, Corollary 3B.7)
\begin{align*}
H^k (X\times X) &=\Hom_K(H_k(X\times X),K) \\
&= \oplus _{i=0} ^{k} {\mathrm{Hom}}_K (H_i (X) \otimes H_{k-i} (X) , K).
\end{align*}
The $i$-th summand in this decomposition contains
$H^i (X) \otimes H^{k-i} (X)$ as a subspace in a canonical way. Moreover,
the $i$-th summand coincides with this tensor product if (and only if)
one of the factors is finite-dimensional. In particular, this is always
the case for $i=0$ and $i=k$. With respect to this decomposition, for any
$\alpha \in H^k (X)$ we have
$$\mu ^{\ast}(\alpha )=1\otimes\alpha+\alpha\otimes 1+r(\alpha ),$$
where $r (\alpha )$, with respect to the above decomposition, has
non-zero coordinates only in summands with $0<i<k$ (\cite{Hatcher},
p.~283).

Let us assume that $A$ is not finite dimensional. Let $d$ be the smallest
integer such that $H^d (X)$ is not finite-dimensional. Consider the
subalgebra $C$ of $A$ that is generated by all elements of degree at most
$d$. For each $\alpha \in H^k (X)$, with $k\leq d$, we have
$r(\alpha) \in \oplus _{i=0} ^{k} {\mathrm{Hom}}_K (H_i (X) \otimes
H_{k-i} (X) , K)=\oplus _{i=1} ^{k-1} H^i (X) \otimes H^{k-i} (X)$,
by the assumption about finite-dimensionality of all $H^i (X)$, with
$i<d$.

In particular, $\mu^{\ast} (\alpha) $ is contained in the subalgebra
$C\otimes C \subset A\otimes A \subset H^{\ast} (X\times X)$. Since $\mu$
is a homomorphism, and $C$ is by assumption generated by such elements
$\alpha$ of degree at most $d$, we conclude that
$\mu^{\ast} (C) \subset C\otimes C$. Thus $C$ is a Hopf algebra. As we
have observed, $C$ is finite dimensional. This is in contradiction to our
assumption on $H^d (X)$.
\end{proof}

\vskip .5cm

\section{Equivariant cohomology and dimension} \label{sec:techprop}
We are going to prove \pref{prop:eqcohinjection} in this section.
First, we recall some notations and basics about equivariant cohomology.
We refer the reader to \cite{AlldayPuppe} for more details. In this
section all cohomologies are considered with $\mathbb Q$-coefficients.

For a compact Lie group $G$ let $BG$ denote the classifying space of $G$
with the classifying principal $G$-bundle $EG \to BG$, such that $EG$ is
contractible.

Let $X$ be a topological space with a $G$-action. The
\emph{Borel construction} of the $G$-action on $X$ is the space
$X_G := (X \times EG) /G$, where $G$ acts diagonally. The $G$-equivariant
cohomology of the space $X$ is by definition
$H^{\ast} _G (X):= H^{\ast} (X_G)$.

Let $T$ be a maximal torus of $G$ and let $N(T)$ be its normalizer in
$G$. Then we have a canonical isomorphism
$H_G ^{\ast} (X) = H_{N(T)} ^{\ast} (X)$ (cf.~\cite{AlldayPuppe},
p.~205-206). In fact, in \cite{AlldayPuppe} this result is explained only
in the case of a connected group $G$ (the only case we need), but the
same proof applies to disconnected groups as well.

We recall the notion of \emph{Krull dimension}, (cf.~\cite{Eisenbud},
Section 8) which we will only use in the case of commutative finitely
generated $\mathbb Q$-algebras. Namely, the Krull dimension of such an
algebra $A$ is the maximal number $s$ such that $A$ contains a polynomial
algebra on~$s$ variables (the usual definition is equivalent to this one
due to \emph{Noether's normalization lemma}). The Krull dimension does
not change if $A$ is replaced by a larger algebra $A'$, such that $A'$ is
a finite module over $A$. In particular, for a finite group~$\Gamma$
acting on $A$ by homomorphisms, the set of fixed points $A^{\Gamma}$ is a
finitely generated $\mathbb Q$-algebra of the same Krull dimension as $A$
(this is a theorem of Noether \cite{Noether}, cf.~\cite{Eisenbud},
Exercise 13.2 and Exercise 13.3).

Now, \pref{prop:eqcohinjection} is a direct consequence of the following
more general result, which is an analogue of Theorem 7.7 in
\cite{Quillen} for $p=0$.

\begin{prop} \label{lastprop}
Let $X$ be a smooth manifold with a smooth action of a compact Lie group
$G$. Assume that $H^{\ast} _G (X)$ is a finitely generated
$\mathbb Q$-algebra. If the Krull dimension of the even-dimensional part
of $H^{\ast} _G (X)$ is $q$, then there is a point $x\in X$ whose
stabilizer has rank $q$.
\end{prop}

\begin{proof}
We have seen that, by replacing $G$ by the normalizer of the maximal
torus, we may assume that the connected component $G^0$ of $G$ is a
torus. The algebra $H^{\ast} _G (X)$ is the set of fixed points of
$H^{\ast} _{G^0 } (X)$ under the canonical action of the finite group
$G/G^0 =\pi_0 (G)$. Hence, $H^{even} _{G^0} (X)$ is a finitely generated
$\mathbb Q$-algebra of Krull dimension~$q$. Thus, we may assume that $G$
is a torus.

For any fixed degree $\ast$, the space $H^{\ast} _G (X)$ is
finite-dimensional. From the principal fibration $X\times EG \to X_G$ we
deduce that in each degree~$\ast$, the cohomology
$H^{\ast} (X\times EG) =H^{\ast} (X)$ is finite dimensional
(\cite{HatcherSpSeq}, Lemma 1.9, for the Serre class $\mathcal C$ of
abelian groups, whose tensor product with $\mathbb Q$ has finite rank
over $\mathbb Q$). Since $X$ is a finite-dimensional space, the total
cohomology ring $H^{\ast} (X)$ is finite dimensional.

Thus, $G$ is a torus and $H^{\ast} (X)$ is finite-dimensional. In this
case the claim is well-known and appears for instance in
\cite{AlldayPuppe}, p.~257, under the additional assumption that the set
of connected components of isotropy groups of the action is finite. (Note
that the Krull dimension of $H^{even} _G (X)$ coincides with the Krull
dimension of $H^{\ast} _G(X)$ used in \cite{AlldayPuppe}, since
$H^{even} _G (X)$ is a finite algebra over the ring
$H^{\ast} (BG) /{\mathrm{ann}} (H^{\ast} _G (X)) $, which is used there
to define the dimension.) This assumption is verified due to
\lref{lastlemma} below.
\end{proof}

\begin{lem} \label{lastlemma}
Let $G$ be a torus and let $X$ be a smooth manifold with a smooth action
of $G$. Assume that $H^k _G (X)$ is finite-dimensional for each $k$. Then
the set of connected components of isotropy groups of the action of $G$
on $X$ is finite.
\end{lem}

\begin{proof}
Assume the contrary. Then we find infinitely many subtori $T_i$ of $T$,
together with some connected components $X_i\neq \emptyset$ of the set of
fixed points $X^{T_i}$, such that no point in $X_i$ is fixed by a torus
larger than $T_i$; the $T_i$ are not necessarily distinct, but the~$X_i$
can be chosen not to intersect. Then, in a small neighborhood $U_i$ of
$X_i$ the set of connected components of isotropy groups is finite. The
equivariant Thom class $u_i$ of~$X_i$ is a nonzero cohomology class in
$H^\ast_G(X)$ of degree equal to the codimension of $X_i$ in $M$, and
with support on $X_i$, thus its restriction to $X\setminus U_i$ is zero.
Therefore, the elements~$u_i$ are linearly independent. Hence, there
exists $k\leq \dim X$ such that $H^k_G(M)$ is infinite dimensional,
contradicting the finiteness assumption.
\end{proof}

\vskip 1.5cm

\end{document}